\documentclass[reqno,a4paper]{amsart}  
\usepackage{enumerate,amsmath,amssymb}     
\allowdisplaybreaks
\usepackage{hyperref}
\usepackage{tikz}
\usepackage[normalem]{ulem}
\usepackage{xcolor}
\theoremstyle{plain}      
 
\newtheorem{thm}{Theorem}[section]     
\newtheorem{theorem}[thm]{Theorem}     
     
\newtheorem{corollary}[thm]{Corollary}     
     
\newtheorem{lemma}[thm]{Lemma}     
\newtheorem{prop}[thm]{Proposition}     
\newtheorem{proposition}[thm]{Proposition}     
     
\theoremstyle{remark}      
 
\newtheorem{examples}[thm]{Examples}
\newtheorem{remark}[thm]{Remark}

\theoremstyle{definition}

\def\be{{\beta}}

\def\Si{{\Sigma}}

\def\epsilon{{\varepsilon}}

\def\phi{{\varphi}}

\DeclareMathAlphabet{\doba}{U}{msb}{m}{n}

\def\ev{{\mathop{\rm ev}\nolimits}}

\let\<\langle 
\let\>\rangle

\newcommand{\definedas}{\mathrel{\raise.095ex\hbox{\rm :}\mkern-5.2mu=}}
     
\newcommand{\rquot}[2]{\raisebox{0.5ex}{$#1$}\!/\!\raisebox{-0.5ex}{$#2$}}
\definecolor{darkgreen}{HTML}{009900}
\definecolor{lightblue}{HTML}{00AAEE}
\definecolor{darkblue}{HTML}{002299}
\renewcommand\sout{\bgroup\markoverwith%
{\textcolor{red}{\rule[0.7ex]{3pt}{1.4pt}}}\ULon}

\def \be{\begin{eqnarray*}}
\def \ee{\end{eqnarray*}}
\def \ben{\begin{enumerate}}
\def \een{\end{enumerate}}
\def \beit{\begin{itemize}}
\def \eeit{\end{itemize}}
\def \bui#1#2{\mathrel{\mathop{\kern 0pt#1}\limits^{#2}}}
\def \buil#1#2{\mathrel{\mathop{\kern 0pt#1}\limits_{#2}}}
\def \bfll{\begin{flushleft}}
\def \efll{\end{flushleft}}
\def \bflr{\begin{flushright}}
\def \eflr{\end{flushright}}

\def \lra{\longrightarrow}

\def \wit{\widetilde}

\def \R{\mathbb{R}}

\begin{document}     
%%%%%%%%%%%%%%%%%%%%%%%%%%%%%%%%%%%%%%%%%%%%%%%%%%%%%%%%%%%%%%%%%%%%%%%%%

%%%%%%%%%%%%%%%%%%%%%%%%%%%%%%%%%%%%%%%%%%%%%%%%%%%%%%%%%%%%%%%%%%%%%%%%%
%  \begin{center}
%  \framebox{\framebox{
%  \vbox{This is project {\red \Project}\\
%  Current version {\blue\Version}, from
%  {\blue\Datum}, most recent changes by {\blue\Person}.}
%  }}
%  \end{center}
%%%%%%%%%%%%%%%%%%%%%%%%%%%%%%%%%%%%%%%%%%%%%%%%%%%%%%%%%%%%%%%%%%%%%%%%%

\title{Some examples of Dirac-harmonic maps}
%\sout{Examples of Dirac-harmonic maps after Jost-Mo-Zhu}
\author{Bernd Ammann} 
\address{Fakult\"at f\"ur Mathematik \\ 
Universit\"at Regensburg \\
93040 Regensburg \\  
Germany}
\email{bernd.ammann@mathematik.uni-regensburg.de}

\author{Nicolas Ginoux} 
\address{Universit\'e de Lorraine, CNRS, IECL, F-57000 Metz, France}
\email{nicolas.ginoux@univ-lorraine.fr}

\begin{abstract}
We discuss a method to construct Dirac-harmonic maps developed by 
J.~Jost, X.~Mo and M.~Zhu in \cite{JostMoZhu}.
The method uses harmonic spinors and twistor spinors, and mainly applies to 
Dirac-harmonic maps of codimension $1$ with target spaces of constant sectional curvature.
Before the present article, it remained unclear 
when the conditions of the theorems in \cite{JostMoZhu} were fulfilled. 
We show that for isometric immersions into spaceforms, these conditions are fulfilled 
only under special assumptions.
In several cases we show the existence of solutions.
\end{abstract}

%% Old \subjclass[2000]{35J60 (Primary), 35P30, 58J50, 58C40 (Secondary)}
% 
% 35J60   Nonlinear PDE of elliptic type
% 35P30   Nonlinear eigenvalue problems, nonlinear spectral theory for PDO
% 57R65   Surgery and handlebodies
% 58J50   Spectral problems; spectral geometry; scattering theory
% 58C40   Spectral theory; eigenvalue problems

\date{\today}

\keywords{Dirac harmonic maps, twistor spinors} 
\setcounter{tocdepth}{1}
\maketitle

\tableofcontents

%%%%%%%%%%%%%%%%%%%%%%%%%%%%%%%%%%%%%%%%%%%%%%%%%%%%%%%%%%%%%%%%%%%%%%%%%
\section{Introduction and main results}\label{s:intromainresults}
%%%%%%%%%%%%%%%%%%%%%%%%%%%%%%%%%%%%%%%%%%%%%%%%%%%%%%%%%%%%%%%%%%%%%%%%%

Let $(M^m,g)$ and $(N^n,h)$ be Riemannian manifolds of dimension $m$ and $n$.
We assume that $M$ carries a fixed spin structure.
Note that in general we do not require $M$ and $N$ to be complete. 
%\sout{be a non-necessarily closed $m$-dimensional Riemannian spin manifold and an $n$-dimensional Riemannian manifold respectively.}
Denote by $\Sigma M$ the corresponding spinor bundle of $M$.
Given a smooth map $f\colon M\lra N$, one can define the twisted Dirac-operator $D^f:=\sum_{j=1}^m e_j\cdot\nabla_{e_j}^{\Sigma M\otimes f^*TN}$ acting on $C^\infty(M,\Sigma M\otimes f^*TN)$, where $(e_j)_{1\leq j\leq m}$ is a local orthonormal frame on $M$ and ``$\,\cdot\,$'' denotes Clifford multiplication $T^*M\otimes\Sigma M\otimes f^*TN\lra \Sigma M\otimes f^*TN$.
Here $\Sigma M\otimes f^*TN$ is to be understood as the real tensor product of $\Sigma M$ with $f^*TN$ and is endowed with a natural Hermitian inner product $\langle\cdot\,,\cdot\rangle$ making the Clifford action of each tangent vector skew-Hermitian.\\
A pair $(f,\Phi)\in C^\infty(M,N)\times C^\infty(M,\Sigma M\otimes f^*TN)$ is called a
\emph{Dirac-harmonic map} if and only if the identities
\begin{equation}\label{eq:defDiracharm}
\left|\begin{array}{ll}D^f\Phi&=0\\ \mathrm{tr}_g(\nabla df)&=\frac12 V_{\Phi}\end{array}\right. 
\end{equation}
hold on $M$, where $V_{\Phi}\in C^\infty(M,f^*TN)$ is the section of $f^*TN$ defined by requiring
\begin{equation}\label{V.def}
h(V_{\Phi},Y):=\sum_{j=1}^m\langle e_j\cdot R_{Y,f_*e_j}^N\Phi,\Phi\rangle\textrm{ for all }Y\in f^*TN.
\end{equation}
Recall that, since the Clifford multiplication of each tangent vector to $M$ and the curvature tensor $R^N$ of $(N,h)$ act in a skew-Hermitian (resp. skew-symmetric) way, the sum $\sum_{j=1}^m\langle e_j\cdot R_{Y,f_*e_j}^N\Phi,\Phi\rangle$ is real.
Here and in the following the notation $e_j\cdot R_{Y,f_*e_j}^N\Phi$ stands for $(e_j\cdot\otimes R_{Y,f_*e_j}^N)\Phi$.
Our convention for curvature tensors is $R_{X,Y}^N=[\nabla_X^N,\nabla_Y^N]-\nabla_{[X,Y]}^N$ for all tangent vectors $X,Y$.\\

In the recent years, there was a considerable interest for 
Dirac-harmonic maps in geometric analysis.
The original motivation for studying Dirac-harmonic maps comes from physics: Dirac-harmonic maps 
are the fermionic analogue of the harmonic map equation.
While harmonic maps are stationary points of the classical (bosonic) energy functional 
$f\mapsto \frac12 \int_M |df|^2\, dv^M$, Dirac harmonic maps are stationary points of the functional $(f,\Phi)\mapsto \frac12 \int_M (|df|^2 +\<\Phi,D^f\Phi\>)\,dv^M$ which is interpreted as the fermionic counterpart of the classical energy functional.
In geometric analysis Dirac-harmonic maps turn out to be an interesting area of investigation,
as on the one hand side these equations are simple enough to allow regularity statements, removal of singularities, short-time existence of associated parabolic flows and much more, and on the other they are involved enough to exhibit a rich structure.\\

The goal of the article \cite{JostMoZhu}, written by J.~Jost, X.~Mo and M.~Zhu was to find solutions $(f,\Phi)$ to the Dirac-harmonic map equations (\ref{eq:defDiracharm}) in the form $(f,\Phi)$ where
% \sout{In \cite{JostMoZhu}, look for solutions}
% \sout{Dirac-harmonic-map-equations}
\begin{equation}\label{eq:dhspecialform} 
\Phi:=\sum_{j=1}^m e_j\cdot\psi\otimes f_*e_j+\phi\otimes\nu,\end{equation}
such that $\psi,\phi\in C^\infty(M,\Sigma M)$ are untwisted spinor fields and such that $\nu\in C^\infty(M,f^*TN)$ is a vector field standing orthogonally onto $df(TM)=f_*(TM)$
%\footnote{B2N: Wir nennen das Differential von $f$ immer $f_*$. Wir müssen die $f_*$-Notation deswegen hier zumindest erwaehnen. Ich waere auch mit $f_*(TM)$, das heisst ohne $df(TM)$ einverstanden. Ja (Nicolas)}
at each point.
The first motivation for considering Dirac-harmonic maps in the form (\ref{eq:dhspecialform}) is that it gives a simple way to produce Dirac-harmonic maps when $M$ is a surface: if we assume $m=2$, that the map $f$ is harmonic, that $\psi$ is a twistor spinor, $\nu=0$, and $\varphi=0$, then the pair $(f,\Phi)$ is a Dirac-harmonic map, see \cite[Theorem 2]{JostMoZhu} and Corollary \ref{c:2dimtwsp} below.
In particular, a lot of examples of Dirac-harmonic maps can be exhibited when $M$ is conformally equivalent to an open subset of one of the model surfaces $\mathbb{S}^2$, $\R^2$ or $\mathbb{H}^2$ or conformally equivalent to a torus with trivial spin structure, we refer to Examples \ref{ex:dhtwistorspinordim2} for more details.
On the other hand closed hyperbolic surfaces do not carry nontrivial twistor spinors and therefore this ansatz is not sufficient to construct Dirac-harmonic maps on such surfaces.\\

As in \cite{JostMoZhu}, we mainly focus in this article on the particular situation where $f$ is an isometric immersion and $n=m+1$.
This allows us --- as usual for isometric immersions --- to identify $TM$ with a subbundle of $f^*TN$.
The above ansatz then easily generates nontrivial Dirac-harmonic maps with harmonic mapping component $f$ out of parallel spinors on $M^m$, see Proposition \ref{p:exdhhypspaceforms}.
However, the existence of nontrivial parallel spinors is very restrictive, making those examples actually very special.
To get examples of Dirac-harmonic maps with non-harmonic mapping component, we furthermore assume $N$ to be oriented with constant sectional curvature $c\in\R$.
Note that in this case the orientations of $M$ and $N$ induce a global smooth unit normal vector field $\nu$ on $f(M)$; any sign convention for the choice of $\nu$ can be used, but should be fixed throughout the article. 
Denote by $W:=-\nabla^N\nu$ the corresponding shape operator of the immersed hypersurface $M$ and by $H:=\frac{1}{m}\mathrm{tr}(W)$ its mean curvature.

\noindent We first characterize Dirac-harmonic maps of the form \eqref{eq:dhspecialform} in that setting (compare \cite[Thm.~1]{JostMoZhu}):

\begin{theorem}\label{t:dhhypspaceforms}
Let $f\colon M^m\lra N^{m+1}$ be an isometric
immersion from a connected Riemannian spin manifold $(M^m,g)$ into an oriented Riemannian manifold $(N^{m+1},h)$ with constant sectional curvature $c\in\R$.
Let $\nu\in \Gamma(f^*TN)$ be a unit normal vector field  of $f(M)\subset N$ with shape operator $W$ and mean curvature $H$ as explained above.
For $\psi,\phi\in C^\infty(M,\Sigma M)$ let $\Phi:=\sum_{j=1}^m e_j\cdot\psi\otimes e_j+\phi\otimes\nu$, where $(e_j)_{1\leq j\leq m}$ is any local orthonormal frame on $M$.
We assume that $\Phi$ does not vanish everywhere.
\beit\item[$i)$] If $m=2$, then $(f,\Phi)$ is a Dirac-harmonic map if and only if $H=0$, $D_M\phi=0$, $c\cdot\Re e(\langle\psi,\phi\rangle)=0$, and $e_1\cdot\nabla_{e_1}^{\Sigma M}\psi-e_2\cdot\nabla_{e_2}^{\Sigma M}\psi=\kappa_1\phi$, where $We_1=\kappa_1 e_1$.
(The vector $e_1$ is a pointwise eigenvector for $W$ associated to the principal curvature $\kappa_1$, we do not require $e_1$ to depend continuously on the basepoint.)
\item[$ii)$] If $m\geq 3$ \and $f$ is a totally umbilical immersion, then $(f,\Phi)$ is a Dirac-harmonic map if and only if $H=-c\Re e(\langle\psi,\phi\rangle)$, $D_M\phi=mH\psi$, $D_M\psi=-\frac{mH}{m-2}\phi$ and $P\psi=0$.
If furthermore $M$ is closed, then $(f,\Phi)$ is a Dirac-harmonic map  if and only if $W=0$, $D_M\phi=0$, $\nabla^{\Sigma M}\psi=0$, and $c\cdot\Re e(\langle\psi,\phi\rangle)=0$. 
%\footnote{Der triviale Fall $\phi=\psi=0$ war hier nicht richtig behandelt. In diesem Fall war die Aussage ``$(f,\Phi)$ is a Dirac-harmonic map with $\Phi\neq 0$'' falsch, aber ``$W=0$, $D_M\phi=0$, $\nabla^{\Sigma M}\psi=0$ and $c\cdot\Re e(\langle\psi,\phi\rangle)=0$.'' wahr. Bist du einverstanden, so wie es jetzt ist? Ja, v\"ollig, Du hast recht, die Annahme $\Phi\neq0$ war nicht n\"otig. (Nicolas) \red{Bernd: Naja, ich habe die  Annahme $\Phi\neq0$ nach oben in die allgemeinen Voraussetzungen verschoben. Wenn wir sie nicht brauchen, koennen wir sie von dort loeschen, oder? Von mir aus kann sie aber auch gerne drin bleiben.} } 
\eeit
\end{theorem}

\begin{remark}\label{rem.H.const}
Note that the assumptions that $f$ is totally umbilical in $ii)$ already implies that $H$ is constant, or more generally: any $m(\geq 2)$-dimensional 
totally umbilical hypersurface in an Einstein manifold has constant mean curvature. This is an elementary consequence of $\delta W=-mdH+\mathrm{Ric}^N(\nu)^T$, which itself follows from the Codazzi-Mainardi-identity (the $1$-form $\mathrm{Ric}^N(\nu)^T\in T^*M$ is defined by $\mathrm{Ric}^N(\nu)^T(X)=h(\mathrm{Ric}^N(\nu),X)$ for all $X\in TM$).
In particular, in the case $m\geq 3$ the existence of a Dirac-harmonic map $(f,\Phi)$ with $\Phi\neq0$ given by $\psi$ and $\phi$ as in the above theorem implies $D_M^2\psi=-\frac{m^2H^2}{m-2}\psi$ and $D_M^2\phi=-\frac{m^2H^2}{m-2}\phi$.
\end{remark}

\begin{remark}
If we allowed for $\phi=\psi=0$, then this theorem would reduce to the classical fact that an isometric immersion is harmonic, if and only if the image has vanishing mean cuvature.
\end{remark}

\begin{remark}
% \sout{At this point we notice a difference with \cite[Thm.~1]{JostMoZhu}: even in the case $m\geq 3$ the authors assume the spinor field $\phi$ to be harmonic, i.e. $D_M\phi=0$.}
Our above theorem also shows that the conditions in \cite[Thm.~1]{JostMoZhu} are very restrictive in the case $m\geq 3$: the authors assume the spinor field $\phi$ to be harmonic, i.e. $D_M\phi=0$.
In this case, Theorem \ref{t:dhhypspaceforms} yields $H=0$ and $c\Re e(\langle\psi,\phi\rangle)=0$. 
Furthermore, $D_M\psi=0$ and $P\psi=0$ and this implies $\nabla^{\Sigma M}\psi=0$. 
As $M$ is isometrically immersed into $N$ with $W=H\cdot\mathrm{Id}=0$, it is a totally geodesic immersion, and $0=mH\cdot\nu=\mathrm{tr}_g(\nabla df)$.
In particular $f$ is harmonic, so no example with non-harmonic map $f$ can be produced. 
Assuming $\psi\not\equiv 0$, these conditions imply that $M$ is Ricci-flat, of special holonomy and $\nabla^{\Sigma M}\phi=0$ as soon as $M$ is closed.
\end{remark}
%\red{The old formulation vaguely indicated that the Jost-Mo-Zhu Theorem 1.1 is not correct, but what we want to say is that their conditions are almost impossible to be satisfied.}

Theorem \ref{t:dhhypspaceforms} allows for producing new explicit examples of Dirac-harmonic maps.
Denote by $N^{m+1}(c)$ any Riemannian spaceform of constant sectional curvature $c$ and by $\wit{N}^{m+1}(c)$ the simply-connected complete Riemannian spaceform of constant sectional curvature $c$.
Replacing the the metric $h$ by $\lambda^2 h$ with a constant $\lambda>0$ does not change the Levi-Civita connection on the tangent bundle of the target,
and thus $D^f$ is unchanged as well.
In the case $c\neq 0$ we can achieve by such a rescaling with $\lambda:= \sqrt{|c|}$ that the rescaled metric has sectional curvature $\pm 1$. Thus we can assume without loss of generality  $c\in\{-1,0,1\}$, i.\thinspace e.\ $\wit{N}^{m+1}(c)=\mathbb{H}^{m+1}(-1)$, $\R^{m+1}$ and $\mathbb{S}^{m+1}(1)$ for $c=-1$, $0$ and $1$ respectively.

\begin{theorem}\label{t:dhtotumbilicalnonharmonic}
Let $f\colon M^m\to \wit{N}^{m+1}(c)$ be a non-minimal totally umbilical isometric immersion from a connected $m\geq3$-dimensional Riemannian spin manifold into $\wit{N}^{m+1}(c)$ for some $c\in\{-1,0,1\}$.
Then there exists a not identically vanishing $\Phi$ in the form \eqref{eq:dhspecialform} such that $(f,\Phi)$ is a Dirac-harmonic map if and only if 
%\blue{up to isometry}\footnote{Bernd: Reinserted! Nicolas: warum? Im Beweis zeigen wir, dass $f(M)$ eine offene Teilmenge eines totalumbilischen $\mathbb{H}^m(-\frac{4}{m+2})$ sein muss -- hierbei bezeichnet dies jedes totalumbilisch eingebettete hyp. Raumes konstanter Schnittkr. $-\frac{4}{m+2}$ in $\mathbb{H}^{m+1}(-1)$. Anwendung einer Isometrie schickt dies wieder auf eine solche Hyperfl.} 
$f(M)$ is an open subset of an umbilic hyperplane $\mathbb{H}^m(-\frac{4}{m+2})$ in $\wit{N}^{m+1}(-1)=\mathbb{H}^{m+1}(-1)$.
\end{theorem}

The following is then an immediate consequence of the theorem by applying the theorem to the lift $\tilde f:\wit{M}\to\mathbb{H}^{m+1}(-1)$. 
\begin{corollary}\label{c:dhtotumbilicalnonharmonic}
Let $f\colon M^m\to N^{m+1}$ be a non-minimal totally umbilical isometric immersion from a connected $m\geq3$-dimensional Riemannian spin manifold into a complete connected manifold $N$ of constant sectional curvature $c$. 
We assume that $(f,\Phi)$ is a Dirac-harmonic map, where
$\Phi$ is in the form \eqref{eq:dhspecialform}.
By Theorem \ref{t:dhtotumbilicalnonharmonic} we know that $c<0$ and by rescaling the metric on $N$ we can achieve $c=-1$.
Let $\tilde f:\wit{M}\to \wit{N}=\wit{N}^{m+1}(-1)=\mathbb{H}^{m+1}(-1)$ be a lift of $f$ to the universal covers.
Then up to isometry $\tilde f(\wit{M})$ is an open subset of a hyperplane $\mathbb{H}^m(-\frac{4}{m+2})$ in $\mathbb{H}^{m+1}(-1)$.
\end{corollary}

\noindent Note that any connected totally umbilical isometrically immersed hypersurface in $\mathbb{H}^{m+1}(-1)$ is an open subset of some $M^m(\kappa)$ with $\kappa\geq-1$.
Here $M^m(\kappa)$ is the canonically embedded complete hypersurface of constant curvature $\kappa$ in $\mathbb{H}^{m+1}(-1)$, that is, $M^m(\kappa)$ is
\begin{itemize}
 \item $\mathbb{H}^m(\kappa)$ for $\kappa\in[-1,0)$, 
 \item a horosphere $\mathbb{R}^m$ if $\kappa=0$, 
 \item the boundary of a geodesic ball if $\kappa>0$.
\end{itemize}

The nontrivial statement in Corollary \ref{c:dhtotumbilicalnonharmonic} is that only the first case can arise and that the value of $\kappa$ is $-\frac{4}{m+2}$.\\

These notes started in 2011 as an informal comment to the authors of \cite{JostMoZhu} in order to lay the basis for our article \cite{AmmannGinoux2011}.
The original title was ``\emph{Examples of Dirac-harmonic maps after Jost-Mo-Zhu}''.
As these informal notes were cited by several authors, we decided in 2018 to transform them into a proper publication made accessible to everyone.

%%%%%%%%%%%%%%%%%%%%%%%%%%%%%%%%%%%%%%%%%%%%%%%%%%%%%%%%%%%%%%%%%%%%%%%%%
\section{Proof of main results}\label{s:proof}
%%%%%%%%%%%%%%%%%%%%%%%%%%%%%%%%%%%%%%%%%%%%%%%%%%%%%%%%%%%%%%%%%%%%%%%%%

The proof starts with two calculations of central importance.
Denote by 
  $$D_M:=\sum_{j=1}^m e_j\cdot\nabla_{e_j}^{\Sigma M}\colon C^\infty(M,\Sigma M)\to C^\infty(M,\Sigma M)$$ 
  the classical Dirac operator by Atiyah and Singer and by 
  $$P\colon C^\infty(M,\Sigma M)\rightarrow C^\infty(M,T^*M\otimes\Sigma M),\quad \psi\mapsto\nabla^{\Sigma M}\psi+\frac{1}{m}\cdot\sum_{j=1}^m e_j^\flat\otimes e_j\cdot D_M\psi$$ 
the Penrose (or twistor) operator on $M$.

\begin{lemma}\label{l:dheqgeneral-one}
With the above notations, one has for $f$ and $\Phi$ given by \eqref{eq:dhspecialform}
%\footnote{\red{Do we also need that~\eqref{eq:defDiracharm} holds? This was not so clear in the old formulation.}}
\be 
D^f\Phi&=&\sum_{j=1}^m\left(\frac{2-m}{m}e_j\cdot D_M\psi-2P_{e_j}\psi\right)\otimes f_*e_j-\psi\otimes\mathrm{tr}_g(\nabla df)\\
& &{}+(D_M\phi)\otimes\nu+\sum_{j=1}^me_j\cdot\phi\otimes\nabla_{e_j}^N\nu.
\ee
\end{lemma}

\begin{proof}
We set $\Psi:=\sum_{j=1}^me_j\cdot\psi\otimes f_*e_j$ and compute
\begin{align*}
D^f\Psi=&\sum_{j=1}^me_j\cdot\nabla_{e_j}^{\Sigma M\otimes f^*TN}(\sum_{k=1}^me_k\cdot\psi\otimes f_*e_k)\\
=&\sum_{j,k=1}^m\Bigl(e_j\cdot \nabla_{e_j}^Me_k\cdot\psi\otimes f_*e_k+e_j\cdot e_k\cdot\nabla_{e_j}^{\Sigma M}\psi\otimes f_*e_k\\
 &\phantom{\sum_{j,k=1}^m}+e_j\cdot e_k\cdot\psi\otimes\nabla_{e_j}^{f^*TN}f_*e_k\Bigr)\\
=&{}-\sum_{j,k=1}^me_k\cdot e_j\cdot\nabla_{e_j}^{\Sigma M}\psi\otimes f_*e_k-2\sum_{j,k=1}^m\underbrace{g(e_j,e_k)}_{=\delta_{jk}}\nabla_{e_j}^{\Sigma M}\psi\otimes f_*e_k\\
 &{}+\sum_{j,k=1}^m e_j\cdot e_k\cdot\psi\otimes(\nabla df)(e_j,e_k)\\
 &{}+\underbrace{\sum_{j,k=1}^m\Bigl(e_j\cdot \nabla_{e_j}^Me_k\cdot\psi\otimes f_*e_k+e_j\cdot e_k\cdot\psi\otimes f_*(\nabla_{e_j}^Me_k)\Bigr)}_{0}.
\end{align*}
Here we used $\sum_{k=1}^m\Big(e_j\cdot \nabla_{e_j}^Me_k\cdot\psi\otimes f_*e_k+e_j\cdot e_k\cdot\psi\otimes f_*(\nabla_{e_j}^Me_k)\Bigr)=0$. This can either be seen by taking a frame 
with $\nabla_{e_j}^Me_k|_p=0$ for some $p\in M$ and all $j,k$, and to calculate in $p$, or alternatively 
by writing $\nabla_{e_j}^Me_k=\sum_{\ell=1}^m\Gamma_{jk}^\ell e_\ell$ and using $\Gamma_{jk}^\ell=-\Gamma_{j\ell}^k$.
Furthermore for $j\neq k$ the expression $e_j\cdot e_k\cdot\psi$ is antisymmetric for permuting $j$ and $k$ while $(\nabla df)(e_j,e_k)$ is symmetric, thus all terms 
$e_j\cdot e_k\cdot\psi\otimes(\nabla df)(e_j,e_k)$ cancel for $j\neq k$. We continue the computation:
\begin{align*}
D^f\Psi\;=\;&{}-\sum_{k=1}^me_k\cdot D_M\psi\otimes f_*e_k-2\sum_{k=1}^m\nabla_{e_k}^{\Sigma M}\psi\otimes f_*e_k-\underbrace{\sum_{k=1}^m \psi\otimes(\nabla df)(e_k,e_k)}_{=\psi\otimes\mathrm{tr}_g(\nabla df)}\\
%=&{}-\sum_{j=1}^me_j\cdot D_M\psi\otimes f_*e_j-2\sum_{j=1}^mP_{e_j}\psi\otimes f_*e_j+\frac{2}{m}\sum_{j=1}^me_j\cdot D_M\psi\otimes f_*e_j\\
% &{}-\psi\otimes\mathrm{tr}_g(\nabla df)\\
\;=\;&\frac{2-m}{m}\sum_{k=1}^me_k\cdot D_M\psi\otimes f_*e_k-2\sum_{k=1}^mP_{e_k}\psi\otimes f_*e_k-\psi\otimes\mathrm{tr}_g(\nabla df),
\end{align*}
where we used the definition of $P$.
On the other hand,
\be 
D^f(\phi\otimes\nu)&=&\sum_{j=1}^me_j\cdot\nabla_{e_j}^{\Sigma M\otimes f^*TN}(\phi\otimes\nu)\\
&=&\sum_{j=1}^me_j\cdot\left(\nabla_{e_j}^{\Sigma M}\phi\otimes\nu+\phi\otimes\nabla_{e_j}^{f^*TN}\nu\right)\\
&=&(D_M\phi)\otimes\nu+\sum_{j=1}^me_j\cdot\phi\otimes\nabla_{e_j}^N\nu.
\ee
Using $D^f\Phi=D^f\Psi+D^f(\phi\otimes\nu)$ this yields the claimed formula for $D^f\Phi$.
\end{proof}

\begin{lemma}\label{l:dheqgeneral-two}
Again we use the above notations, we assume that $\Phi$ is given by~\eqref{eq:dhspecialform} and that $V_{\Phi}$ is given by~\eqref{V.def}.
Then we have for all $Y\in f^*TN$,
$$
h(V_{\Phi},Y)=2\sum_{j,k=1}^m h(R_{Y,f_*e_j}^Nf_*e_k,\nu)\,\Re e(\langle e_j\cdot e_k\cdot\psi,\phi\rangle).
$$ 
\end{lemma}

\begin{proof}As in the proof of Lemma \ref{l:dheqgeneral-two}, let $\Psi:=\sum_{j=1}^me_j\cdot\psi\otimes f_*e_j$ and $V_\Psi$ be the associated vector field as in \eqref{V.def}. 
Recall that $\Phi\mapsto\sum_{j=1}^me_j\cdot R_{Y,f_*e_j}^N\Phi$ is Hermitian, in particular
\[h(V_{\Phi},Y)=h(V_{\Psi},Y)+h(V_{\phi\otimes\nu},Y)+2\Re e\left(\sum_{j=1}^m\langle e_j\cdot R_{Y,f_*e_j}^N\Psi,\phi\otimes\nu\rangle\right)\]
for all $Y\in f^*TN$.
We compute each term separately.
First,
\be 
h(V_{\Psi},Y)&=&\sum_{j,k,\ell=1}^m\Re e\left(\langle e_j\cdot R_{Y,f_*e_j}^N(e_k\cdot\psi\otimes f_*e_k),e_\ell\cdot\psi\otimes f_*e_\ell\rangle\right)\\
&=&\sum_{j,k,\ell=1}^m\Re e\left(\langle (e_j\cdot e_k\cdot \psi)\otimes R_{Y,f_*e_j}^N f_*e_k,e_\ell\cdot\psi\otimes f_*e_\ell\rangle\right)\\
&=&\sum_{j,k,\ell=1}^m h(R_{Y,f_*e_j}^N f_*e_k,f_*e_\ell)\,\Re e\left(\langle e_j\cdot e_k\cdot \psi,e_\ell\cdot\psi\rangle\right)\\
&=&-\sum_{j,k,\ell=1}^m h(R_{Y,f_*e_j}^N f_*e_k,f_*e_\ell)\,\Re e\left(\langle e_\ell\cdot e_j\cdot e_k\cdot \psi,\psi\rangle\right).
\ee
This sum contains different kind of terms. For with $j=k$ we calculate
  $$\Re e\left(\langle e_\ell\cdot e_j\cdot e_j\cdot \psi,\psi\rangle\right)=-\Re e\left(\langle e_\ell\cdot \psi,\psi\rangle\right)=0$$
and with similar arguments all terms with $k=\ell$ or $j=\ell$ vanish, including $j=k=\ell$.
%In the case that $i$, $j$ and $k$ are pairwise different, one has  $$ e_\ell\cdot e_j\cdot e_k=  e_k\cdot e_\ell\cdot e_j = e_j \cdot e_k\cdot e_\ell.$$
Given a fixed triple $(j,k,\ell)$ with $j\neq k\neq \ell\neq j$, consider the $\rquot{\mathbb{Z}}{3\mathbb{Z}}$-action given by the cyclic permutation sending $j$ on $k$ and $k$ on $\ell$.
Then the sum corresponding to the $\rquot{\mathbb{Z}}{3\mathbb{Z}}$-orbit vanishes: by definition of the Clifford multiplication,
$$ e_\ell\cdot e_j\cdot e_k=  e_k\cdot e_\ell\cdot e_j = e_j \cdot e_k\cdot e_\ell,$$
so that, using the first Bianchi identity for the curvature tensor of $(N,h)$,
\begin{align*}
&\;\sum_{\sigma\in\rquot{\mathbb{Z}}{3\mathbb{Z}}}h(R_{Y,f_*e_{\sigma(j)}}^N f_*e_{\sigma(k)},f_*e_{\sigma(\ell)})\,\Re e\left(\langle e_{\sigma(\ell)}\cdot e_{\sigma(j)}\cdot e_{\sigma(k)}\cdot \psi,\psi\rangle\right)\\
= &\;\left(h(R_{Y,f_*e_j}^N f_*e_k,f_*e_\ell)+h(R_{Y,f_*e_k}^N f_*e_\ell,f_*e_j)+h(R_{Y,f_*e_\ell}^N f_*e_j,f_*e_k)\right)\\
 &\;\quad\cdot \Re e\left(\langle e_\ell\cdot e_j\cdot e_k\cdot \psi,\psi\rangle\right)\\
=&\;0.
\end{align*}
Therefore, $V_{\Psi}=0$.

For $\phi\otimes\nu$, using $h(R_{Y,f_*e_j}^N\nu,\nu)=0$, we obtain
\be 
h(V_{\phi\otimes\nu},Y)&=&\sum_{j=1}^m\Re e\left(\langle (e_j\cdot\phi)\otimes R_{Y,f_*e_j}^N\nu,\phi\otimes\nu\rangle\right)\\
&=&\sum_{j=1}^mh(R_{Y,f_*e_j}^N\nu,\nu)\,\Re e\left(\langle e_j\cdot\phi,\phi\rangle\right)\\
&=&0,
\ee
so that $V_{\phi\otimes\nu}=0$.
As for the cross term, we obtain
\be 
\Re e\left(\sum_{j=1}^m\langle e_j\cdot R_{Y,f_*e_j}^N\Psi,\phi\otimes\nu\rangle\right)&=&\sum_{j,k=1}^m\Re e\left(\langle (e_j\cdot e_k\cdot\psi)\otimes R_{Y,f_*e_j}^Nf_*e_k,\phi\otimes\nu\rangle\right)\\
&=&\sum_{j,k=1}^m h(R_{Y,f_*e_j}^Nf_*e_k,\nu)\,\Re e\left(\langle e_j\cdot e_k\cdot\psi,\phi\rangle\right).
\ee
The result follows.
\end{proof}

As a straightforward consequence of Lemmata~\ref{l:dheqgeneral-one} and~\ref{l:dheqgeneral-two}, we reprove \cite[Theorem 2]{JostMoZhu} by J.~Jost, X.~Mo and M.~Zhu. 
%\sout{\cite[Thm.~2]{JostMoZhu} obtain the}

\begin{corollary}\label{c:2dimtwsp}
Let $m=2$, assume that the spinor field $\psi$ is a twistor spinor, that~$\phi$ is the zero section and that the map $f$ is harmonic.
Then $(f,\Phi)$, defined by \eqref{eq:dhspecialform}, is a Dirac-harmonic map.
\end{corollary}

\begin{proof}
Lemma~\ref{l:dheqgeneral-one} implies $D^f\Phi=0$.
Furthermore $V_{\Phi}$ vanishes since the Hermitian inner product $\langle e_j\cdot e_k\cdot\psi,e_l\cdot\psi\rangle$ is purely imaginary for all $j,k,l\in\{1,2\}$.
\end{proof}

\begin{examples}\label{ex:dhtwistorspinordim2}\
{\rm\begin{enumerate}
\item The two-dimensional round sphere $\mathbb{S}^2$ carries a $4$-dimensional space of twis\-tor spinors; a twistor spinor on $\mathbb{S}^2$ is the sum of a $\frac{1}{2}$- and of a $-\frac{1}{2}$-Killing spinor, see e.g. \cite{BFGK} or \cite[App. A]{Ginoux09}.
By Corollary \ref{c:2dimtwsp}, for any harmonic map $f\colon\mathbb{S}^2\to N^n$, there exists a nonzero $\Phi\in C^\infty(\mathbb{S}^2,\Sigma \mathbb{S}^2\otimes f^*TN)$ such that $(f,\Phi)$ is a Dirac-harmonic map.
\item If $M^2$ is the flat plane $\R^2$ or any open subset of it, then it carries an infinite-dimensional space of twistor spinors; this space turns out to be isomorphic to the sum of the space of holomorphic functions with that of anti-holomorphic functions on $M$, see e.g. \cite[Prop. A.2.3]{Ginoux09}.
As a consequence of Corollary \ref{c:2dimtwsp}, for any harmonic map $f\colon M^2\to N^n$, there exists a nonzero $\Phi\in C^\infty(M,\Sigma M\otimes f^*TN)$ such that $(f,\Phi)$ is a Dirac-harmonic map.
\item Since the kernel of the Penrose operator as well as harmonicity of $f$ are conformally invariant, the former for every $m$, see \cite{BFGK}, and the latter only for $m=2$, see \cite{EellsSampson64}, the examples described above are still valid when the metric $g$ is chosen in the conformal class of the standard metric on $M$.
In particular, the same kind of examples can be built on the hyperbolic plane $\mathbb{H}^2$ since it is conformally equivalent to a flat disk.
\item Nontrivial quotients of model surfaces may carry twistor spinors, this depends on the group that is divided out but also on the spin structure chosen on the quotient.
For instance, the only compact quotients of model surfaces carrying nontrivial twistor spinors are $M=\mathbb{S}^2$ and $M=\mathbb{T}^2$ where the latter carries the trivial spin structure, that is, the spin structure that is a trivial $2$-fold covering of the unit circle bundle over $M$. This spin structure can also be characterized as the only spin structure on $\mathbb{T}^2$ which is not obtained by restricting a spin structure on a solid torus to its boundary torus $\mathbb{T}^2$.
Note in particular that no nontrivial twistor spinor exists on closed hyperbolic surfaces, thus no example of the form above can be produced in that case.
\end{enumerate}}
\end{examples}

It is interesting to notice another consequence of Lemmata \ref{l:dheqgeneral-one} and \ref{l:dheqgeneral-two}:
\begin{proposition}\label{p:exdhhypspaceforms}
With the above notations, assume that $M^m$ carries a nontrivial parallel spinor.
Then for any Riemannian manifold $N^{m+1}$ and any harmonic map $f\colon M^m\to N^{m+1}$, there exists a non-identically-vanishing $\Phi\in C^\infty(M,\Sigma M\otimes f^*TN)$ such that $(f,\Phi)$ is a Dirac-harmonic map.
\end{proposition}
\begin{proof}
We let $\psi$ be a nonzero parallel spinor on $M$, $\varphi=0$ and define $\Phi$ as in \eqref{eq:dhspecialform}.
Because of $\phi=0$, $D_M\psi=0$ and $P\psi=0$ (any parallel spinor is both a harmonic spinor and a twistor spinor) as well as $\mathrm{tr}_g(\nabla df)=0$ since $f$ is harmonic, we have $D^f\Phi=0$.
On the other hand, because of $\phi=0$, we have $V_\Phi=0$ by Lemma \ref{l:dheqgeneral-two}, so that $\frac{V_{\Phi}}{2}=0=\mathrm{tr}_g(\nabla df)$.
Thus $(f,\Phi)$ is a Dirac-harmonic map.
\end{proof}

We now reformulate Lemmata~\ref{l:dheqgeneral-one} and~\ref{l:dheqgeneral-two} when $f\colon M^m\to N^n$ is an isometric immersion, $n=m+1$, the manifold $N^n$ is oriented. 
Again we identify in this case $TM$ with a subbundle of $f^*TN$.
Further let $\nu$ be the unit normal vector field induced by the orientations of $M$ and $N$.

\begin{prop}\label{p:dheqhypspaceforms}
Assume $f$ is an isometric immersion from $M^m$ into an oriented Riemannian manifold $N^{m+1}$, that $\Phi$ and $V_\Phi$ are given by \eqref{eq:dhspecialform} and \eqref{V.def} respectively, where $\nu$ is the unit normal vector field induced by the orientations of $M$ and $N$.
Then one has
\be 
D^f\Phi&=&\sum_{j=1}^m\left(\frac{2-m}{m}e_j\cdot D_M\psi-2P_{e_j}\psi-We_j\cdot\phi\right)\otimes e_j\\
& &{}+(D_M\phi-mH\psi)\otimes\nu.
\ee
Moreover, if $N$ has constant sectional curvature $c\in\R$, then $V_{\Phi}=-2mc\Re e\left(\langle\psi,\phi\rangle\right)\nu$.
\end{prop}

\begin{proof}The identification mentioned above yields $f_*e_j=e_j$.
Using $\nabla df=W\otimes\nu$, one has $\mathrm{tr}_g(\nabla df)=\mathrm{tr}(W)\nu=mH\nu$. 
Moreover, since $\nabla_X^N\nu=-WX$ and $W$ is symmetric, Lemma~\ref{l:dheqgeneral-one} gives
\be 
D^f\Phi&=&\sum_{j=1}^m\left(\frac{2-m}{m}e_j\cdot D_M\psi-2P_{e_j}\psi\right)\otimes e_j-mH\psi\otimes\nu\\
& &{}+(D_M\phi)\otimes\nu-\sum_{j=1}^me_j\cdot\phi\otimes We_j\\
&=&\sum_{j=1}^m\left(\frac{2-m}{m}e_j\cdot D_M\psi-2P_{e_j}\psi\right)\otimes e_j+(D_M\phi-mH\psi)\otimes\nu\\
& &-\sum_{j,k=1}^mg(We_j,e_k)\,e_j\cdot\phi\otimes e_k\\
&=&\sum_{j=1}^m\left(\frac{2-m}{m}e_j\cdot D_M\psi-2P_{e_j}\psi\right)\otimes e_j+(D_M\phi-mH\psi)\otimes\nu\\
& &{}-\sum_{k=1}^m W(e_k)\cdot\phi\otimes e_k\\
&=&\sum_{j=1}^m\left(\frac{2-m}{m}e_j\cdot D_M\psi-2P_{e_j}\psi-We_j\cdot\phi\right)\otimes e_j\\
& &{}+(D_M\phi-mH\psi)\otimes\nu,
\ee
which proves the first identity.
Assume now that $(N^{m+1},h)$ has constant sectional curvature $c$.
Then the curvature tensor of $N$ is given for all $X,Y,Z,T\in TN$ by $h(R_{X,Y}^NZ,T)=c\cdot\left(h(X,T)h(Y,Z)-h(X,Z)h(Y,T)\right)$.
Lemma~\ref{l:dheqgeneral-two} yields for all $Y\in f^*TN$:
\be 
h(V_{\Phi},Y)&=&2\sum_{j,k=1}^m h(R_{Y,e_j}^Ne_k,\nu)\,\Re e(\langle e_j\cdot e_k\cdot\psi,\phi\rangle)\\
&=&2c\cdot\sum_{j,k=1}^m\left(h(Y,\nu)\underbrace{h(e_j,e_k)}_{\delta_{jk}}-h(Y,e_k)\underbrace{h(e_j,\nu)}_{0}\right)\Re e(\langle e_j\cdot e_k\cdot\psi,\phi\rangle)\\
&=&-2mc\,h(\nu,Y)\,\Re e(\langle\psi,\phi\rangle).
\ee
Thus 
$$h(V_{\Phi}+2mc\, \Re e(\langle\psi,\phi\rangle)\,\nu,Y)=0\qquad\forall\, Y\in f^*TN$$
which concludes the proof.
\end{proof}

\noindent Now we prove Theorem \ref{t:dhhypspaceforms}.
As $f$ is isometric,$\nabla df$ is the vector-valued second fundamental form of $f(M)$ in $N$, and we have $\mathrm{tr}_g(\nabla df)=mH\cdot \nu$.
Proposition~\ref{p:dheqhypspaceforms} implies that $(f,\Phi)$ is a Dirac-harmonic map if and only if $D_M\phi=mH\psi$, 
\begin{equation}\label{eq:ejPej.neu}
 \frac{2-m}{m}e_j\cdot D_M\psi-2P_{e_j}\psi-We_j\cdot\phi=0\text{ for all }1\leq j\leq m
\end{equation}
and $mH\cdot\nu=\frac{V_{\Phi}}{2}=-mc\Re e\left(\langle\psi,\phi\rangle\right)\nu$.
% \sout{In other words, $(f,\Phi)$ is a Dirac-harmonic map if and only if $D_M\phi=mH\psi$, $H=-c\Re e\left(\langle\psi,\phi\rangle\right)$ and}
% \begin{equation}\label{eq:ejPej}\sout{\frac{2-m}{m}X\cdot D_M\psi-2P_X\psi-WX\cdot\phi=0}\end{equation}
% \sout{for all $X\in TM$.}
Note that taking the Clifford product of $e_j$ with \eqref{eq:ejPej.neu} and summing over $j$ gives, using the symmetry of $W$,
\begin{eqnarray}\label{eq:DMpsiphi}
\nonumber0&=&\frac{2-m}{m}\sum_{j=1}^me_j\cdot e_j\cdot D_M\psi-2\underbrace{\sum_{j=1}^me_j\cdot P_{e_j}\psi}_{0}-\sum_{j=1}^m e_j\cdot We_j\cdot\phi\\
&=&(m-2)D_M\psi+mH\phi.
\end{eqnarray}
{\it Case $m=2$}: Then it follows from (\ref{eq:DMpsiphi}) that $H\phi=0$.
Since on the open set $\Omega:=\{x\in M\,|\,H(x)\neq 0\}$ the spinor $\phi$ has to vanish, so does $\psi$ on $\Omega$ because of $D_M\phi=mH\psi$, so that $\Phi=0$ on $\Omega$ and therefore on $M$ by the unique continuation property for elliptic self-adjoint differential operators.
Since we look for a pair $(f,\Phi)$ with $\Phi\neq 0$, we necessarily have $\Omega=\varnothing$, that is, $H=0$ on $M$.
The identities $D_M\phi=mH\psi$, $H=-c\Re e\left(\langle\psi,\phi\rangle\right)$ become $D_M\phi=0$ and $c\Re e\left(\langle\psi,\phi\rangle\right)=0$ respectively.
Taking the Clifford product of $e_j$ with \eqref{eq:ejPej.neu} and recalling the definition of $P$, one obtains 
\be 
e_j\cdot We_j\cdot\phi&=&-2e_j\cdot P_{e_j}\psi\\
&=&-2e_j\cdot\nabla_{e_j}^{\Sigma M}\psi+D_M\psi
\ee
for both $j\in\{1,2\}$.
The difference of this equation for $j=1$ and the one for $j=2$ yields $e_2\cdot We_2\cdot\phi-e_1\cdot We_1\cdot\phi=2(e_1\cdot\nabla_{e_1}^{\Sigma M}\psi-e_2\cdot\nabla_{e_2}^{\Sigma M}\psi)$.
Take now $(e_j)_{1\leq j\leq 2}$ to be a pointwise orthonormal basis of $T_xM$ made of eigenvectors for $W$ for some fixed $x\in M$.
With the condition $H=0$ one can write $We_1=\kappa_1e_1$ and $We_2=-\kappa_1 e_2$, therefore one obtains
\begin{equation}\label{eq.fin}
2(e_1\cdot\nabla_{e_1}^{\Sigma M}\psi-e_2\cdot\nabla_{e_2}^{\Sigma M}\psi)=2\kappa_1\phi.
\end{equation}
As \eqref{eq.fin} implies \eqref{eq:ejPej.neu} trivially, 
this shows $i)$.\\

{\it Case $m\geq 3$}: It follows from (\ref{eq:DMpsiphi}) that $D_M\psi=-\frac{mH}{m-2}\phi$.
As a consequence, the assumption $W=H\cdot\mathrm{Id}$ (total umbilicity of $f$) makes \eqref{eq:ejPej.neu} equivalent to $P\psi=0$.
This proves the general case.
We now specialize to the case that $M$ is closed.
Then $D_M^2\psi=-\frac{m^2H^2}{m-2}\psi$ and $D_M^2\phi=-\frac{m^2H^2}{m-2}\phi$, see Remark~\ref{rem.H.const}.
Since $D_M^2$ is a nonnegative operator, it does not have any negative eigenvalue on a closed manifold, therefore $\psi=\phi=0$ unless $H=0$, which is the only possibility because of $\Phi\neq 0$.
Therefore $H$ --- hence $W$ --- has to vanish on $M$.
Since both $D_M\psi=0$ and $P\psi=0$, one obtains $\nabla^{\Sigma M}\psi=0$ (hence $\psi$ is actually parallel).
This shows $ii)$ and concludes the proof of Theorem \ref{t:dhhypspaceforms}.\\

\noindent We now prove Theorem \ref{t:dhtotumbilicalnonharmonic}.
Let $f\colon M^m\to \wit{N}^{m+1}(c)$ be a totally umbilical immersion with $m\geq 3$ and $W=H\cdot\mathrm{Id}\neq0$. 
Assume the pair $(f,\Phi)$ to be Dirac-harmonic.
Recall that then $M$ has to be noncompact (Theorem \ref{t:dhhypspaceforms}).
Since $P\psi=0$, we know that $D_M^2\psi=\frac{mS_g}{4(m-1)}\psi$, where $S_g$ is the scalar curvature of $(M^m,g)$, see e.g. \cite{BFGK} or \cite[Prop. A.2.1]{Ginoux09}.
Comparing with $D_M^2\psi=-\frac{m^2H^2}{m-2}\psi$ and assuming $\psi\neq 0$ (otherwise $\phi=0$ hence $\Phi=0$, as we have seen above), we obtain $\frac{mS_g}{4(m-1)}=-\frac{m^2H^2}{m-2}$ and the Gau\ss{} equation $S_g=m(m-1)c+m^2H^2-|W|^2=m(m-1)(H^2+c)$ implies $H^2=-\frac{m-2}{m+2}c$, in particular $c$ must be negative, w.l.o.g. $c=-1$.
Therefore $\wit{N}^{m+1}(c)=\mathbb{H}^{m+1}(-1)$.
In that case, $f(M)$ must be an open subset of a totally umbilical (but non-totally geodesic) hyperbolic hyperplane of constant sectional curvature $H^2+c=\frac{4}{m+2}c=-\frac{4}{m+2}<0$.
Up to changing $\nu$ into $-\nu$, one can assume $H$ to be positive, so that $H=\sqrt{\frac{m-2}{m+2}}$.
Now the space of twistor spinors on any hyperbolic space is explicitly known: it is the direct sum of the space of Killing spinors for the opposite (imaginary) Killing constants.
More precisely $\mathrm{ker}(P)=\mathcal{K}_p\oplus\mathcal{K}_m$ on $M$, where $\mathcal{K}_p:=\{\psi\in C^\infty(M,\Sigma M)\,|\,\nabla_X^{\Sigma M}\psi=\frac{i}{\sqrt{m+2}}X\cdot\psi\;\forall X\in TM\}$ and $\mathcal{K}_m:=\{\psi\in C^\infty(M,\Sigma M)\,|\,\nabla_X^{\Sigma M}\psi=-\frac{i}{\sqrt{m+2}}X\cdot\psi\;\forall X\in TM\}$. 
Looking for $\psi$ in the form $\psi=\psi_p+\psi_m$ with \emph{a priori} arbitrary $(\psi_p,\psi_m)\in\mathcal{K}_p\oplus\mathcal{K}_m$, we write the equations of Theorem \ref{t:dhhypspaceforms} down: one has $D_M\psi=-\frac{im}{\sqrt{m+2}}(\psi_p-\psi_m)$, in particular one has to choose $\phi:=-\frac{m-2}{mH}D_M\psi=i\sqrt{m-2}(\psi_p-\psi_m)$.
The formulas for $\psi$ and $\phi$ immediately imply  $D_M\phi=mH\psi$.
The only remaining condition having to be satisfied is $H=-c\cdot\Re e(\langle\psi,\phi\rangle)$, that is,
\be 
\sqrt{\frac{m-2}{m+2}}&=&\sqrt{m-2}\cdot\Re e(-i\langle\psi_p+\psi_m,\psi_p-\psi_m\rangle)\\
&=&\sqrt{m-2}\cdot\Im m(|\psi_p|^2-|\psi_m|^2+\langle\psi_m,\psi_p\rangle-\langle\psi_p,\psi_m\rangle)\\
&=&-2\sqrt{m-2}\cdot\Im m(\langle\psi_p,\psi_m\rangle),
\ee
that is, $\Im m(\langle\psi_p,\psi_m\rangle)=-\frac{1}{2\sqrt{m+2}}$.
Note that the inner product $\langle\psi_p,\psi_m\rangle$ is anyway constant on $M$ (its first derivative vanishes).
Evaluation at a point $x\in M$ yields linear maps $\ev_x^p\colon\mathcal{K}_p\to \Sigma_xM$ and $\ev_x^m\colon\mathcal{K}_m\to \Sigma_xM$ that are both injective (an imaginary Killing spinor is a parallel section w.r.t. a modified connection) and surjective (the hyperbolic space has the maximal possible number of imaginary Killing spinors).
Let $\tilde \psi_p:=\ev_x^p(\psi_p)$ and $\tilde \psi_m:=\ev_x^m(\psi_m)$. 
So in order to classify all admissible pairs $(\psi_p,\psi_m)$ it is sufficient to classify all pairs $(\tilde\psi_p,\tilde\psi_m)$ in $\Sigma_xM$ with $\Im m(\langle\tilde\psi_p,\tilde\psi_m\rangle)=-\frac{1}{2\sqrt{m+2}}$.
This is easy: for each non-zero $\tilde\psi_p\in \Si_xM$, let $\tilde\psi_m:=\frac{i}{2\sqrt{m+2}|\tilde\psi_p|^2}\tilde\psi_p$, then $\Im m(\langle\tilde\psi_p,\tilde\psi_m\rangle)=-\frac{1}{2\sqrt{m+2}}$ and obviously all admissible pairs are of the form $(\tilde\psi_p,\frac{i}{2\sqrt{m+2}|\tilde\psi_p|^2}\tilde\psi_p+\chi)$ where $\chi$ runs over the real hyperplane of $\Si_xM$ defined by the equation $\Im m(\langle\tilde\psi_p,\chi\rangle)=0$.

This concludes the proof of Theorem \ref{t:dhtotumbilicalnonharmonic}.\\

\section{Concluding remarks}

\noindent It may be interesting to know whether $2$-dimensional examples with $\phi\neq 0$ can be obtained from Theorem~\ref{t:dhhypspaceforms}.
%\red{Bernd: Oder meinst du ``from Proposition~\ref{p:exdhhypspaceforms}''? Nicolas: ich meinte wohl Theorem 1.1}.
Namely if one considers the Clifford torus $M^2:=\mathbb{S}^1(\frac{1}{\sqrt{2}})\times\mathbb{S}^1(\frac{1}{\sqrt{2}})$ sitting canonically in $N:=\mathbb{S}^3$, then the inclusion map is minimal (with principal curvatures $1$ and $-1$), but the following short argument shows that the only Dirac-harmonic maps $(f,\Phi)$ in the form (\ref{eq:dhspecialform}) have vanishing $\phi$-component.
Note at first that on the flat two-torus $M$ the Schr\"odinger-Licherowicz formula implies using $D_M\phi=0$ that $\phi$ is parallel.
Thus the statement is immediate if $M$ carries one of the three spin structures that do not allow for nonzero parallel spinor fields on the two-torus $M$. 
But even if the spin structure on $M^2$ is the one admitting parallel spinors, then any Dirac-harmonic map $(f,\Phi)$ in the form (\ref{eq:dhspecialform}) must have $\phi=0$ due to the following reason. 
We know from Theorem \ref{t:dhhypspaceforms} that $(f,\Phi)$ is a Dirac-harmonic map if and only if $H=0$ (which is the case here), $D_M\phi=0$, $c\cdot\Re e(\langle\psi,\phi\rangle)=0$, and $e_1\cdot\nabla_{e_1}^{\Sigma M}\psi-e_2\cdot\nabla_{e_2}^{\Sigma M}\psi=\kappa_1\phi$, where $We_1=\kappa_1 e_1$.
As mentioned above $D_M\phi=0$ is equivalent to $\phi$ being parallel.
But, taking into account that, in the particular example of the embedding $M^2\hookrightarrow\mathbb{S}^3$, the principal curvature $\kappa_1$ is constant and the vector fields $e_1,e_2$ are globally defined and \emph{parallel} on $M^2$, we have, differentiating w.r.t. $e_1$:
$$0=e_1\cdot\nabla_{e_1}^{\Si M}\nabla_{e_1}^{\Si M}\psi-e_2\cdot\nabla_{e_1}^{\Si M}\nabla_{e_2}^{\Si M}\psi$$
and in the same way $0=e_1\cdot\nabla_{e_2}^{\Si M}\nabla_{e_1}^{\Si M}\psi-e_2\cdot\nabla_{e_2}^{\Si M}\nabla_{e_2}^{\Si M}\psi$.
By $R^{\Si M}=0$ and $[e_1,e_2]=0$, we have $\nabla_{e_2}^{\Si M}\nabla_{e_1}^{\Si M}\psi=\nabla_{e_1}^{\Si M}\nabla_{e_2}^{\Si M}\psi$, so that 
\be 
0&=&e_1\cdot\nabla_{e_1}^{\Si M}\nabla_{e_1}^{\Si M}\psi-e_2\cdot\nabla_{e_2}^{\Si M}\nabla_{e_1}^{\Si M}\psi\\
&=&e_1\cdot\nabla_{e_1}^{\Si M}\nabla_{e_1}^{\Si M}\psi-e_2\cdot(-e_1\cdot e_2\cdot\nabla_{e_2}^{\Si M}\nabla_{e_2}^{\Si M}\psi)\\
&=&e_1\cdot\left(\nabla_{e_1}^{\Si M}\nabla_{e_1}^{\Si M}\psi+\nabla_{e_2}^{\Si M}\nabla_{e_2}^{\Si M}\psi\right)\\
&=&e_1\cdot(\nabla^{\Si M})^*\nabla^{\Si M}\psi\qquad\textrm{since }\nabla_{e_i}e_i=0,
\ee
so that $(\nabla^{\Si M})^*\nabla^{\Si M}\psi=0$, that is, $\nabla^{\Si M}\psi=0$.
In turn, this implies $\kappa_1\phi=0$ and therefore $\phi=0$ because of $\kappa_1\neq0$.
Actually we have shown that $(f,\Phi)$ in the form (\ref{eq:dhspecialform}) is a Dirac-harmonic map if and only if $\phi=0$ and $\psi$ is parallel.

% \noindent Note that, in case $N=\mathbb{H}^{m+1}(-1)$, we have actually shown in the proof of Theorem \ref{t:dhtotumbilicalnonharmonic} that the example described is the \emph{only} one with $(f,\Phi)$ in the particular form (\ref{eq:dhspecialform}).
In the case where $m=2$ no non-trivial example of Dirac-harmonic maps from a closed hyperbolic surface can be obtained with Corollary \ref{c:2dimtwsp}, since those do not carry non-zero twistor spinors.
% \red{Nicolas: Verstehe ich auch nicht so recht. Ich dachte $\psi$ muss ein twistor Spinor sein, falls $\phi=0$ ist. Aber es könnte ja auch ein Paar mit $\phi\not\equiv 0$ geben. Bernd: die Aussage ist, es gibt keine dirac-harmonischen Abbildungen gibt, die von dem Korollar kommen. Und alle vom Korollar haben $\phi=0$.
% Nicolas: ... \"ubrigens sehe ich nicht, wie man minimale isometrische Immersionen von geschlossenen hyperbolischen Fl\"achen in Raumformen erstellt (bei Lawson ist die von der minimalen Immersion $M^2\to\mathbb{S}^3$ induzierte Metrik {\sl a priori} nur konform zur urspr\"unglichen Metrik, nicht?), dazu kommt, dass harmonische Spinoren auf dieser Fl\"ache existieren m\"ussen und das auch ist nicht garantiert (Nicolas). Bernd: Ich weiß auch nicht, ob das geht. Aber wenn eben bereits eine Voraussetzung nicht mehr gilt, dann funktioniert die ganze Methode nicht. Mehr haben wir nicht behauptet.}
In that setting, examples can be produced with the help of index-theoretical methods, see e.g. \cite{AmmannGinoux2011}. 
Curvature conditions implying the vanishing of the $\Phi$ defined in (\ref{eq:dhspecialform}) have been investigated by X. Mo \cite{Mo10} and confirm that only few examples of that special form can be expected.\\

\noindent For higher codimensions the same approach can probably be carried out, the existence of a global unit normal $\nu$ already restricting the generality.
On the other hand, there are in that case obvious examples of Dirac-harmonic maps which are \emph{not} in the form (\ref{eq:dhspecialform}): take e.g. $M:=\mathbb{S}^2=\mathbb{C}\mathrm{P}^1$ embedded totally geodesically into $N=\mathbb{C}\mathrm{P}^2$, then we know by the index theorem (see e.g. \cite{AmmannGinoux2011}) that $\mathrm{dim}_{\mathbb{C}}(\mathrm{ker}(D^f))\equiv 2\;(4)$ and is at least $4$-dimensional by \cite{JostMoZhu} (the space of twistor spinors on $\mathbb{S}^2$ injects into $\mathrm{ker}(D^f)$), so that it is at least - actually exactly - $6$-dimensional. 
Now if $\Phi\in\mathrm{ker}(D^f)$, then it is an easy remark that w.r.t. the canonical splitting $\Phi=\Phi_++\Phi_-$ one has $D^f\Phi_\pm=0$ and $V_{\Phi_\pm}=0$, in particular $(f,\Phi_+)$ and $(f,\Phi_-)$ are Dirac-harmonic maps; since $\mathrm{dim}_{\mathbb{C}}(\mathrm{ker}(D_\pm^f))\geq 3$ and the space of pure twistor spinors is complex $2$-dimensional, there are at least one non-trivial $\Phi_+\in \mathrm{ker}(D_+^f)$ and one non-trivial $\Phi_-\in \mathrm{ker}(D_-^f)$ such that $(f,\Phi_\pm)$ are Dirac-harmonic but do not come from any twistor spinor on $\mathbb{S}^2$.

\providecommand{\bysame}{\leavevmode\hbox to3em{\hrulefill}\thinspace}


\begin{thebibliography}{10}

\bibitem{AmmannGinoux2011}
B.~Ammann, N.~Ginoux, \emph{Dirac-harmonic maps from index theory}, Calc. Var. Part. Diff. Eq. \textbf{47} (2013), no. 3-4, 739--762.

\bibitem{BFGK}
H.~Baum, T.~Friedrich, R.~Grunewald, I.~Kath, \emph{Twistor and Killing spinors on Riemannian manifolds}, Teubner-Texte zur Mathematik, Band 124, Teubner-Verlag Stuttgart/Leipzig 1991.

\bibitem{EellsSampson64}
J.~Eells, J.H.~Sampson, \emph{Harmonic mappings of Riemannian manifolds}, Amer. J. Math. \textbf{86} (1964), 109--160.

\bibitem{Ginoux09}
N.~Ginoux, \emph{The Dirac spectrum}, Lecture Notes in Mathematics 1976, Springer, 2009.

\bibitem{JostMoZhu}
J.~Jost, X.~Mo, M.~Zhu, \emph{Some explicit constructions of Dirac-harmonic maps}, J. Geom. Phys.  \textbf{59} (2009), no. 11, 1512--1527.

\bibitem{Mo10}
X.~Mo, \emph{Some rigidity results for Dirac-harmonic maps}, Publ. Math. Debrecen \textbf{77} (2010), no. 3-4, 427--442.

\end{thebibliography}
\end{document}